\documentclass[12pt,a4paper]{amsart}
\usepackage{amsmath,amssymb,amsfonts, float}
\usepackage[all]{xy}
\usepackage{caption}
\usepackage{enumerate}
\usepackage{mathpazo}
\usepackage{a4wide}
\usepackage{comment}

\usepackage{bm}
\usepackage{graphicx}
\usepackage[breaklinks=true]{hyperref}

\usepackage{xcolor}
\usepackage{multicol}
\usepackage{tabularx}

\usepackage{hyperref}
\usepackage[capitalize]{cleveref}

\usepackage[normalem]{ulem}

\newtheorem{thm}{Theorem}[section] 
\newtheorem*{thm*}{Theorem} 
\newtheorem{prop}[thm]{Proposition}
\newtheorem{lem}[thm]{Lemma}
\newtheorem{cor}[thm]{Corollary}

\theoremstyle{definition}
\newtheorem{definition}[thm]{Definition}
\newtheorem{expl}[thm]{Example}

\newtheorem{question}[thm]{Question}
\newtheorem{rem}[thm]{Remark}

\DeclareMathOperator{\C}{\mathbb{C}}
\DeclareMathOperator{\Z}{\mathbb{Z}}

\DeclareMathOperator{\F}{\mathbb{F}}

\DeclareMathOperator{\Q}{\mathbb{Q}}

\DeclareMathOperator{\Gal}{\text{Gal}}
\DeclareMathOperator{\Hom}{{\rm Hom}}

    \DeclareFontFamily{U}{wncy}{}
    \DeclareFontShape{U}{wncy}{m}{n}{<->wncyr10}{}
    \DeclareSymbolFont{mcy}{U}{wncy}{m}{n}
    \DeclareMathSymbol{\Sha}{\mathord}{mcy}{"58}

\numberwithin{equation}{section}

\DeclareSymbolFont{bbold}{U}{bbold}{m}{n}
\DeclareSymbolFontAlphabet{\mathbbold}{bbold}

\usepackage{hyperref}

\newcommand{\Rad}{{\rm Rad}}

\newcommand{\Ann}{{\rm Ann}}
\newcommand{\Sp}{{\mathcal{S}}}

\newcommand{\rad}{{\rm rad}}

\let\emptyset\varnothing

\title{Supercharacters of finite abelian groups \\ and applications to spectra of \\  $U$-unitary Cayley graphs}
 \author{Tung T. Nguyen, Nguy$\tilde{\text{\^{e}}}$n Duy T\^{a}n }

 \address{Department of Mathematics and Computer Science, Lake Forest College, Lake Forest, Illinois, USA}
 \email{tnguyen@lakeforest.edu}
 
  \address{
Faculty Mathematics and Informatics, Hanoi University of Science and Technology, 1 Dai Co Viet Road, Hanoi, Vietnam } 
\email{tan.nguyenduy@hust.edu.vn}

\thanks{TTN is partially supported by an AMS-Simons Travel Grant.  NDT is partially supported by the Vietnam National
Foundation for Science and Technology Development (NAFOSTED) under grant number 101.04-2023.21}
\begin{document}
\begin{abstract}
We define super-Cayley graphs over a finite abelian group $G$. Using the theory of supercharacters on $G$, we explain how their spectra can be realized as a super-Fourier transform of a superclass characteristic function. Consequently, we show that a super-Cayley graph is determined by its spectrum once an indexing on the underlying group $G$ is fixed. This generalizes a theorem by Sander-Sander, which investigates the case where $G$ is a cyclic group. We then use our theory to define and study the concept of a $U$-unitary Cayley graph over a finite commutative ring $R$, where $U$ is a subgroup of the unit group of $R$. Furthermore, when the underlying ring is a Frobenius ring, we show that there is a natural supercharacter theory associated with $U$. By applying the general theory of super-Cayley graphs developed in the first part, we explore various spectral properties of these $U$-unitary Cayley graphs, including their rationality and connections to various arithmetical sums.

\end{abstract}

\maketitle

\section{Introduction}
Cayley graphs over finite commutative rings are widely studied in the literature.  
Examples include  Paley graphs, unitary Cayley graphs, $p$-unitary Cayley graphs,  involutory Cayley graphs, cubelike graphs
and various natural generalizations (see \cite{unitary, huang2022quadratic, jones2020isomorphisms, keshavarzi2025involutory, minac2023paley,  podesta2019spectral, podesta2021_finitefield, suntornpoch2016cayley}). We refer the reader to \cite{anderson2021graphs, madhumitha2023graphs} and the extensive references therein for some surveys on this line of research. These studies reveal that Cayley graphs defined over rings often have richer structures than those  defined over an abstract abelian group. In fact, in the former case, the interplay between the multiplicative and additive structures of the underlying ring plays a crucial role in the understanding of the associated graphs. Quite naturally, ideals are fundamental in the analysis of these graphs (see, for example \cite{chudnovsky2024prime, nguyen2024certain, nguyen2024gcd, suntornpoch2016cayley}).  Furthermore, it is often the case that the underlying ring has a simple spectral description, which enables us to explicitly describe the spectra of their associated graphs via some classical arithmetical sums such as Gauss sums and Ramanujan sums (see, for example, \cite{garcia2018supercharacter,jones2020isomorphisms, minac2023paley, nguyen2024gcd, podesta2019spectral}). We remark that these sums manifest in other areas of mathematics as well. For example, they appear naturally in the theory of special values of $L$-functions (see \cite{lemmermeyer}). Additionally, they play a fundamental role in  the study of certain Fekete polynomials (see \cite{chidambaram2023fekete,MTT3, MTT4}); namely, these sums are precisely special values of Fekete polynomials at roots of unity. As we explain later in this article, our investigations further show that these arithmetic sums also appear in the spectral description of certain graphs.

In this article, we introduce the notion of $U$-unitary Cayley graphs that unifies various constructions in the literature, including all the graphs mentioned in the previous paragraph. Furthermore, when the underlying ring is a Frobenius ring, we provide a concrete description for the spectra of these generalized gcd-graphs. We achieve this by developing a theory of super-Cayley graphs over an abstract finite abelian group $G$. It turns out that these spectra can be expressed as certain super-Fourier transforms on the space of superclass functions on $G$. As a consequence of this characterization, we show that a super-Cayley graph over $G$ is determined by its eigenvalues once we fix an indexing of superclasses of $G$. This is a direct generalization of \cite[Theorem 1.2]{sander2015so} for classical gcd-graphs over the cyclic group $G=\Z/n$. See \cite{sander2015so, schlage2021determinant} for two different approaches to this theorem when $G=\Z/n$, and \cite{minavc2024isomorphic} for an analog version when $G$ is a quotient of a polynomial ring over a finite field.

We now summarize our main results. We refer the reader to the main text for precise statements that include the definitions of several mathematical terms. The first theorem concerns the spectra of super-Cayley graphs defined over a finite abelian group. 

\begin{thm} \label{thm:first}
Let $G$ be a finite abelian group. Let $(\mathcal{K}, \mathcal{X})$ be a supercharacter theory on $G.$ Let $\Gamma(G,S)$ be a super-Cayley graph with respect to $(\mathcal{K}, \mathcal{X})$.  Then, the following statements hold. 

\begin{enumerate}
\item The eigenvalues of $\Gamma(G,S)$ are $\{\lambda_i\}$ counted with multiplicity where $\lambda_i = \widehat{1_S}(K_i)$. Here $K_i$ is a class in $\mathcal{K}, $ $1_{S}$ is the characteristic function of $S$ and $\widehat{1_{S}}$ is its super-Fourier transform. 
\item If we fix an indexing of $\mathcal{K}$, then $\Gamma(G,S)$ is determined by its spectrum. 
\end{enumerate}
\end{thm}

The next theorem applies the general results in \cref{thm:first} to the case where the underlying group is the additive group of a finite Frobenius ring. Here, we have more explicit results including the singularity and rationality of the spectra of these super-Cayley graphs. 

\begin{thm}
    Let $R$ be a finite commutative Frobenius ring and $U$ be a subgroup of $R^{\times}.$ Then the following statements hold. 
    \begin{enumerate}
        \item There exists a supercharacter theory $(\mathcal{K}, \mathcal{X})$ associated with $U.$
        \item Let $\Gamma(R,S)$ be a $U$-unitary Cayley graph. Then, each eigenvalue in the spectrum of $\Gamma(R,S)$ can be explicitly expressed as a linear combination of several generalized Ramanujan sums. 

        \item Let $n$ be a positive integer such that $nR=0.$ Let $T$ be a symmetric subset of $R$. Let $K$ be a subfield of $\Q(\zeta_n)$, $H_1$ the subgroup of $\Gal(\Q(\zeta_n)/\Q) = (\Z/n)^{\times}$  associated with $K$ under the Galois correspondence, and $U_1$ the image of $H_1$ under the canonical map $(\Z/n)^{\times} \to R^{\times}.$ Then $\Gamma(R,T)$ is $K$-rational if and only if it is $U_1$-unitary. 

        \item \label{test:condition} Let $U=(R^{\times})^p$ where $p$ is either $1$ or a prime number with $\gcd(p,n)=1$. Suppose further that $\Gamma(R,(R^{\times})^p)$ is connected and anti-connected. Then, $\Gamma(R,(R^{\times})^p)$ is prime if and only if $0$ is not an eigenvalue of $\Gamma(R,R^{\times}).$
    \end{enumerate}
\end{thm}
We remark that we also prove various results about the connectedness and primeness of $U$-unitary Cayley graphs. In addition, we discuss other examples related to the fourth statement. However, these results are rather technical to state, and therefore we refer the readers to the main text for the precise statements. Finally, we note that after the completion of this paper, we found via Google Scholar that part $(3)$ was recently independently obtained by Godsil and Spiga by a different method (see \cite[Theorem 1.3]{godsil2025integral}). While their approach is purely group-theoretic and hence works for all finite groups, ours is more ring-theoretic. More precisely, we exploit the fact that an abstract finite abelian group has an \textit{extra} ring structure and as a result we can use both their additive and multiplicative structures (see \cref{rem:group_ring}). 
\subsection{Outline.}
We develop the theory of super-Cayley graphs in \cref{sec:general_theory}. In this section, we show that the spectra of these graphs can be described as a super-Fourier transform of a superclass characteristic function. In \cref{sec:U_unitary}, we introduce the concept of $U$-unitary Cayley graphs over a finite commutative ring. We discuss various graph-theoretic properties of these graphs, including their connectedness and primeness.  Finally, in \cref{sec:frobenius}, we apply the results from \cref{sec:general_theory} and \cref{sec:U_unitary} to the case where the underlying group is induced by the additive structure of a Frobenius ring. More precisely, we prove that there is a natural supercharacter theory on $R$ for each choice of $U$. Consequently, the spectra of these $U$-unitary Cayley graphs can be expressed as super-Fourier transforms. We study various arithmetical properties of these spectra, including their regularity and rationality. The latter part significantly generalizes the results of \cite{nguyen2024integral, so2006integral} regarding integral graphs.  Additionally, we pose a precise question about the relationship between the primeness of a $U$-unitary Cayley graph and its spectrum. In the last subsection, we explore the connections between these super-Fourier transforms and certain arithmetical sums, including Ramanujan sums, Gauss sums, and Heilbronn sums.

\section{Supercharacter theory and super-Cayley graphs over an abelian group} \label{sec:general_theory}
In this section, we discuss supercharacter theories of an abelian group $G$. We then apply this theory to study certain Cayley graphs on $G.$ In the next section, we will apply this general theory to the case where $G$ is isomorphic to $(R,+)$ where $R$ is a finite commutative ring. We first recall the definition of a supercharacter theory for $G$ (see \cite{brumbaugh2014supercharacters, diaconis2008supercharacters, fowler2014ramanujan}.)  We note that the theory of supercharacter works for a general finite group including a non-abelian one. However, since our primary focus is on Cayley graphs on rings, we will concentrate on the case where $G$ is abelian. We also remark that we add one additional condition for a supercharacter theory, which we will show to be important for later parts.

\begin{definition} \label{def:supercharacter}
    Let $G$ be a finite abelian group. Let $\mathcal{K}=\{K_1, K_2, \ldots, K_m\}$ be a partition of $G$ and $\mathcal{X} =\{X_1, X_2, \ldots, X_m \}$ a partition of the dual group $\widehat{G}=\Hom(G, \C^{\times})$ of characters of $G$. We say that $(\mathcal{K}, \mathcal{X})$ is a supercharacter theory for  $G$ if the following conditions are satisfied 
    \begin{enumerate}
        \item $\{0 \} \in \mathcal{K};$
        \item $|\mathcal{X}| = |\mathcal{K}|;$
        \item For each $X_i \in \mathcal{X}$, the character sum 
        \[ \sigma_i = \sum_{\chi \in X_i} \chi\]
        is constant on each $K \in \mathcal{K}$;
    \item In this paper, as it will be clear later, we add one more condition to $(\mathcal{K}, \mathcal{X})$; namely for a fixed $\chi \in \mathcal{X}$ the sum $\sum_{k \in K_i} \chi(k)$ does not depend on the choice of $\chi \in X.$ 

    \end{enumerate}
\end{definition}
Since our goal is to study the Cayley graphs with respect to the pair $(\mathcal{K}, \mathcal{X})$, we will further assume that $\mathcal{K}$ is symmetric; namely $K_i = -K_i$ for each $1 \leq i \leq m.$

\begin{rem}
    It is unclear to us whether condition $4$ in \cref{def:supercharacter} is a consequence of other conditions. Conditions $3$ and $4$ are somewhat related to the concept on semimagic squares (in \cite[Corollary 2.1.2]{CM1_b}, we show that a normal matrix with constant row sums must also have constant column sums.) 
 \end{rem}

As in \cite{fowler2014ramanujan}, we will denote by  $\sigma_{i} = \sum_{\chi \in X_i} \chi$ and 
by $\sigma_i(K_j)$ the value of $\sigma_i$ at an element of $K_j$. Similarly, we will write 
\[ \Omega_j(\chi) = \sum_{k \in K_j} \chi(k), \]
where $\chi$ is a character of $G$. Furthermore, since $\Omega_j$ is constant when evaluated by elements in $X_i$, we will denote by $\Omega_j(X_i)$ the value $\Omega_j(\chi)$ for an arbitrary $\chi \in X_i$. We have the following duality. 
\begin{prop} \label{prop:relation}
    For each $1 \leq i, j \leq m$ we have 
    \[ \frac{\Omega_j(X_i)|}{|K_j|}  = \frac{\sigma_i(K_j)}{|X_i|}.\]
    Additionally, if $\mathcal{K}$ is symmetric, we also have 
    \[ \Omega_j(X_i)| = \overline{\Omega_j(X_i)|}. \]
\end{prop}

\begin{proof}
    We have 
    \[ \sum_{\chi \in X_i} \sum_{k \in K_j} \chi(k) = \sum_{\chi \in X_i} \left[\sum_{k \in K_j} \chi(k) \right] = |X_i| \Omega_j(X_i).\]
    Similarly 
    \[ \sum_{\chi \in X_i} \sum_{k \in K_j} \chi(k) = \sum_{k \in K_j} \left[\sum_{\chi \in X_i} \chi(k) \right] = |K_j| \sigma_i(K_j).\]
    This shows that $|X_i| \Omega_j(X_i) = |K_j| \sigma_i(K_j)$. For the second statement, since $K_j = -K_j$, we have 
    \[ \Omega_j(X_i) = \sum_{k \in K_j} \chi_i(k) = \sum_{k \in K_j} \chi_i(-k) = \sum_{k \in K_j} \overline{\chi_i(k)} = \overline{\Omega_j(X_i)}.
    \qedhere\]
\end{proof}
We will fix a supercharacter theory $(\mathcal{K}, \mathcal{X})$ for $G$ throughout this section.  Let $S$ be a symmetric subset of $G$; namely $S=-S$ and $0 \not \in S.$  The Cayley graph $\Gamma(G,S)$ is the graph whose vertex set is $G$ and two vertices $u,v \in G$ are adjacent if $u-v \in S.$  Once we fix the pair $(\mathcal{K}, \mathcal{X})$, it is natural to consider Cayley graphs associated with it. Specifically,  in what follows, we will define and develop a spectral theory for super-Cayley graphs for $(\mathcal{K}, \mathcal{X}).$ We first formally introduce this definition.

\begin{definition} \label{defn:super_cayley}
    A Cayley graph $\Gamma(G,S)$ is called a super-Cayley graph with respect to the pair $(\mathcal{K}, \mathcal{X})$ if the generating set $S$ is a union of some classes in $\mathcal{K}.$
\end{definition}

We have the following observation. 
\begin{prop} \label{prop:complement}
    Suppose that $\Gamma(G,S)$ is a super-Cayley graph. Then, the complement of $\Gamma(G,S)$ is also a super-Cayley graph. 
\end{prop}

\begin{proof}
    The complement of $\Gamma(G,S)$ is precisely $\Gamma(G, S^{c})$ where $S^c = G \setminus (\{0\} \cup S).$ Since $S$ is a union of some classes in $\mathcal{K}$ and $\{0\} \in \mathcal{K}$, $S^c$ is a union of some classes in $\mathcal{K}$ as well. We conclude that the complement of $\Gamma(G,S)$ is a super-Cayley graph. 
\end{proof}
We now discuss spectra of super-Cayley graphs. By the circulant diagonalization theorem for finite abelian groups (see \cite{kanemitsu2013matrices}), the spectrum of a Cayley graph $\Gamma(G,S)$ is given by the family 
\[ \left\{ \sum_{s \in S} \chi(s)\right \}_{\chi}, \]
where $\chi$ runs over the dual group $\widehat{G}:= {\rm Hom}(G, \C^{\times})$ of all characters of $G$. When $\Gamma(G,S)$ is a super-Cayley graph, we can write 
\[ \sum_{s \in S} \chi(s) = \sum_{K_j \subset S} \sum_{s\in K_j}\chi(s) = \sum_{K_j \subset S} \Omega_j(\chi).\]
By this computation, we have the following proposition. 

\begin{prop} \label{prop:eigenvalues}
Let $\Gamma(G,S)$ be a super-Cayley graph. Then,  the spectrum of $\Gamma(G,S)$ is the multiset $\{[\lambda_i]_{|X_i|} \}_{1 \leq i \leq m}$ where 
    \[\lambda_i = \sum_{K_j \subset S} \Omega_j(X_i).\]
    Consequently, $\Gamma(G,S)$ has at most $m$ distinct eigenvalues. 
\end{prop}
We will now show that, similar to the classical theory, the eigenvalues $\{\lambda_i\}_{1 \leq i \leq m}$ described in \cref{prop:eigenvalues} can be realized as a super-Fourier transform. To achieve this goal, we  recall some basic background in the representation theory of finite groups. First, we remark that the space of complex valued functions $f\colon G \to \C$ is equipped with the following natural inner product.
\begin{equation} \label{eq:inner_product}
\langle f_1, f_2 \rangle = \frac{1}{|G|} \sum_{g \in G} f_1(g) \overline{f_2(g)}.
\end{equation}

An important feature of $\sigma_i$ is that they are constant in superclasses. For this reason, it is natural to consider functions with a similar property (see \cite{brumbaugh2014supercharacters, fowler2014ramanujan}). 
\begin{definition}
A function $f\colon G \to \C$ is called a superclass function if $f$ is constant on each superclass in $\{K_1, K_2, \ldots, K_m \}.$ We will denote by $f(K_i)$ the value of $f$ at arbitrary element $e \in K_i$ (by definition, this does not depend on the choice of $e$.) The space of all superclass functions with respect to the pair $(\mathcal{K}, \mathcal{X})$ will be denoted by $\mathcal{S}.$
\end{definition}

The space $\Sp$ naturally inherits the inner product described in \cref{eq:inner_product}. For $f_1, f_2 \in \Sp$, their inner product can be rewritten as 

\begin{equation} \label{eq:inner_product_new}
\langle f_1, f_2 \rangle = \frac{1}{|G|} \sum_{\ell=1}^m |K_\ell| f_1(K_\ell) \overline{f_2(K_\ell)}.
\end{equation}

As explained in \cite{brumbaugh2014supercharacters}, $\{\sigma_i \}_{i=1}^n$ forms an orthogonal basis for the space of superclass functions. More precisely, we have the following orthogonal relations. 
\[ \langle \sigma_i, \sigma_j \rangle = |X_i| \delta_{i,j}.\]
Because $\{\sigma_i\}$ forms a basis for $\Sp$, each $f \in \Sp$ can be written as a linear combination of $\{\sigma_i \}_{1 \leq i \leq m}$; namely,
\begin{equation} \label{eq:fourier}
f = \sum_{\ell=1}^m \hat{f}(K_\ell) \sigma_\ell. 
\end{equation}
The map $\mathcal{F}\colon \Sp \to \Sp$ given by $f \mapsto \hat{f}$ is called the non-normalized  discrete Fourier transform of $f$ (see \cite[Equation 3.1]{brumbaugh2014supercharacters}.) This map $\mathcal{F}$ is an automorphism of $\Sp$ since $f$ and $\hat{f}$ determine each other. In fact, using the orthogonality relations on $\{\sigma_i\}_{1 \leq i \leq m}$, we have 
\[ |X_i| \hat{f}(K_i) = \hat{f}(K_i) \langle \sigma_i, \sigma_i \rangle = \langle f, \sigma_i \rangle = \frac{1}{|G|} \sum_{\ell=1}^m |K_\ell| f(K_\ell) \overline{\sigma_i(K_{\ell})}.\] 
By \cref{prop:relation}, we know that $|K_{\ell}|\sigma_i(K_\ell)=|X_i| \Omega_{\ell}(X_i)$ and $\Omega_j(X_i)| = \overline{\Omega_j(X_i)|}$. As a result, we conclude that 
\begin{equation} \label{eq:explicit Fourier}
 \hat{f}(K_i) = \sum_{\ell=1}^m f(K_\ell) \overline{\Omega_\ell(X_i)}= \sum_{\ell=1}^m f(K_\ell) \Omega_\ell(X_i). 
\end{equation}
We are now ready to state our claim that the spectrum of $\Gamma(G,S)$ can be realized as a discrete Fourier transform. 

\begin{thm}
    Let $\Gamma(G,S)$ be a super-Cayley graph. Let $1_{S}$ be the characteristic function of $S$ and $\{[\lambda_i]_{|X_i|} \}_{1 \leq i \leq m}$ the spectrum of $\Gamma(G,S)$ as described in \cref{prop:eigenvalues}. Then for each $1 \leq i \leq m$
    \[ \lambda_i = \widehat{1_S}(K_i).\]
\end{thm}
We have the following corollary, which is a direct generalization of \cite[Theorem 1.2]{sander2015so}. 
\begin{cor}
    $S$ is uniquely determined by the system of eigenvalues $\{\lambda_i\}_{1 \leq i \leq m}.$
\end{cor}

\begin{proof}
    This follows from the fact that the super-Fourier transform map $\mathcal{F}\colon \mathcal{S} \to \mathcal{S}$ is an isomorphism. 
\end{proof}

\section{ $U$-unitary graphs} \label{sec:U_unitary}

\subsection{Definition and examples}
Let $R$ be a finite commutative ring and $S$ be a symmetric subset of $R$; namely $S=-S$ and $0 \not \in S.$ In the previous section, we study Cayley graphs over a finite abelian group. In this section, we apply this general theory to study Cayley graphs on $(R,+)$. More precisely, we are interested in a class of Cayley graphs $\Gamma(R,S)$ over $R$ where the generating set $S$ has an arithmetic origin. Previous works in the literature have explored the case $S$ is the subgroup of invertible elements (see \cite{unitary, klotz2007some}), the subgroup of $p$-powers (see \cite{nguyen2024certain, podesta2019spectral, podesta2021_finitefield}), the solutions of some algebraic equations (see \cite{keshavarzi2025involutory}), and some natural generalizations (see \cite{chudnovsky2024prime, minac2023paley, nguyen2024gcd}).  In those cases, $S$ and hence $\Gamma(R,S)$ often have some additional symmetries that make them more interesting to study.

To start our discussion, we fix a subgroup $U$ of $R^{\times}$ such that $-1 \in U$ (this will ensure that all relevant generating sets are symmetric). The Cayley graph $\Gamma(R,U)$ is introduced in \cite{chudnovsky2024prime} in the context of prime Cayley graphs. More precisely, in that work, we provide the necessary and sufficient conditions for this graph to be prime (see \cite[Theorem 4.1]{chudnovsky2024prime}). We now introduce the concept of generalized gcd-graphs over $R.$
\begin{definition} \label{def:u-unitary}
Let $S$ be a symmetric subset of $R.$ We say that $\Gamma(R,S)$ is a generalized $U$-unitary Cayley graph if $S$ is stable under the action of $U$; namely $US =S.$
\end{definition}
We remark that \cref{def:u-unitary} unifies various concepts in the literature. In the following, we discuss some of them.

\begin{expl} \label{expl:Paley}
Let $R = \mathbb{Z}/p$  where $p$ is a prime number of the form $4k+1$. Let $\rho\colon (\Z/p)^{\times} \to \C^{\times}$ be a quadratic character modulo $p$ given by the Legendre symbol; namely $\rho(m) =\left(\frac{m}{p} \right)$. Let $U$ be the kernel of $\rho$, that is, 
\[ U = \{m \mid \rho(m) = 1. \} \] The Cayley graph $\Gamma(R,U)$ is known as a Paley graph. These Paley graphs have a rich history in mathematics, first appearing implicitly in Paley's 1933 work \cite{paley1933orthogonal} on Hadamard matrices. Interestingly, Carlitz independently rediscovered them in a different mathematical context \cite{carlitz1960theorem}. The significance of Paley graphs in the field is perhaps best summarized by Gareth A. Jones in \cite{jones2020paley}, who noted:
\begin{quotation}
Paley graphs and their automorphism groups are inevitable encounters for anyone deeply engaged in the study of algebraic graph theory or finite permutation groups. 
\end{quotation}
In general, if $R$ is a finite commutative ring and $\rho\colon R \to \C$ is a multiplicative character (see \cite[Section 4.2]{chudnovsky2024prime}), we can define the generalized Paley graph $P_{\rho}$ as the Cayley graph $\Gamma(R, \ker(\rho)).$ The case $R$ is the ring of integers modulo a positive integer $n$ is considered in \cite{minac2023paley}. There, using the theory of Ramanujan and Gauss sums, we describe explicitly the spectra of these Paley graphs. Furthermore, using properties of the  $L$-function attached to $\rho$, we provide some estimates on the Cheeger numbers of $P_{\chi}.$

\end{expl}
\begin{expl}
If $U = R^{\times}$, then a $U$-unitary Cayley graph is exactly a gcd-graph as described in \cite{klotz2007some, nguyen2024gcd, so2006integral}. These gcd-graphs generalize the class of unitary Cayley graphs (corresponding to the case $S = R^{\times}$), which have been a subject of interdisciplinary research since the pioneering work of Klotz and Sander. For example, see \cite{bavsic2009perfect, sander2018structural, so2006integral} for some research directions related to these unitary Cayley graphs.

\end{expl}

\begin{expl} \label{expl:p-unitary}
    Let $U=(R^{\times})^p$ if $p$ is an odd positive integer and $U=  (R^{\times})^p\cup (-(R^\times)^p)$ if $p$ is even. The Cayley graph $\Gamma(R,U)$ is known as a $p$-unitary Cayley graph. As explained in various works such as \cite{podesta2019spectral, podesta2023waring}, these graphs have found applications in solving various Diophantine problems over finite fields. We note that the case $p=2$ gives a generalization of the classical Paley graph (see  \cite{huang2022quadratic,liu2015quadratic} for some investigations regarding these graphs over an arbitrary finite commutative ring).
\end{expl}

\begin{expl} \label{expl:involution}
Let $U \subset R^{\times}$ be the set of elements $x \in R$ such that $x^2=1.$ The graph $\Gamma(R,U)$ is called an involutory Cayley graph (see \cite{keshavarzi2025involutory}).
    \end{expl}

\begin{expl}
Let $R = (\Z/2)^m$, where $m$ is a positive integer. Then $R^{\times} = {(1,1,\ldots, 1)}$, so every subgroup $U$ of $R^{\times}$ is trivial. As a result, $\Gamma(R,S)$ is a $U$-unitary Cayley graph for all choices of $S$. This type of graph is known as a cubelike graph in the literature (see \cite{cheung2011perfect}).

\end{expl}

We now discuss some structural properties of $U$-unitary Cayley graphs. Expanding on Jones' observation in \cref{expl:Paley}, we remark that for a general $U$-unitary Cayley graph $\Gamma(R,S)$, there is a natural group homomorphism from $U$ to the group of automorphisms of $\Gamma(R,S)$. 

\begin{prop} \label{prop:automorphism}
    Let $\Gamma(R,S)$ be a $U$-unitary graph. Then the natural action of $U$ on $R$ induces an automorphism on $\Gamma(R,S).$
\end{prop}

\begin{proof}
Since $U$ is a group and $S$ is stable under $U$, for each $u \in U$, the multiplication by $u$ is an automorphism of $\Gamma(R,S).$
\end{proof}
We now provide some concrete criteria for a Cayley graph $\Gamma(R,S)$ to be a $U$-unitary graph. Let $\mathcal{K} = \{K_1, K_2, \ldots, K_{m} \}$ be the orbits of $R$ under the action of $U.$ By definition, two elements $x,y$ belong to the same orbit if and only if $x=uy$ for some $u \in U.$ Additionally, we know that $S$ is stable under the action of $U$ if and only if $S$ is a union of some classes in  $\mathcal{K}.$ In summary, we have the following characterization of $U$-unitary Cayley graphs. 
\begin{prop}
The following conditions are equivalent.
\begin{enumerate}
    \item $\Gamma(R,S)$ is a $U$-unitary graph. 
    \item For each orbit $K_i$ where $1 \leq i \leq m$, if $S \cap K_i \neq \emptyset$ then $K_i \subset S.$
    \item $S$ is a disjoint union of some orbits $K_i$.
\end{enumerate}
\end{prop}

Using an argument identical to \cref{prop:complement} we have the following.
\begin{prop} \label{prop:complement_U}
    Suppose that $\Gamma(R,S)$ is an $U$-unitary Cayley graph. Then its complement is also a $U$-unitary Cayley graph. 
\end{prop}

\subsection{Connectedness of $U$-unitary Cayley graphs}
In this section, we study the connectedness of a $U$-unitary Cayley graph.  We recall that $\Gamma(R,S)$ is connected if and only if $S$ generates $R$ as an abelian group. We note that, for connectedness, we only need the additive structure of $R.$ However, as we soon see, the multiplicative structure also plays an implicit but fundamental role. This fact echos our earlier comment in the introduction that Cayley graphs over a ring often has richer structure.

We introduce the following convention: for each $1 \leq i \leq m$, let $I_i$ be the ideal generated by an element of $ x \in K_i$ (by definition of $K_i$, $I_i$ is independent of the choice of $x$). 

\begin{prop} \label{prop:connectedness}
Let $\Gamma(R,S)$ be a $U$-unitary Cayley graph. Suppose that $\Gamma(R,U)$ is connected. Then the following conditions are equivalent. 
    \begin{enumerate}
        \item $\Gamma(R,S)$ is connected. 
        \item $R=\sum_{i} I_i$ where the sum is over all ideals $I_i$ such that $S \cap K_i \neq \emptyset.$
    \end{enumerate}
    Furthermore, if one of these conditions are satisfied, then 
    \[ {\rm diam}(\Gamma(R,S)) \leq {\rm diam}(\Gamma(R,U)) t.\]
    Here $t$ is the smallest positive integer in which there exists $i_1, i_2, \ldots, i_t$ such that $R=\sum_{s=1}^t I_{i_s}$ and $S \cap K_{i_s} \neq \emptyset.$
\end{prop}

\begin{proof}
    The proof for this statement is almost identical to the one given in \cite[Proposition 3.1]{nguyen2024gcd}. For the sake of completeness, we provide the details below.

    We first shows that $(1) \implies (2).$ In fact, since $\Gamma(R,S)$ is connected, we can write $1 = \sum_{i} a_i s_i,$ 
    where $a_{i} \in \Z$ and $s_{i} \in S.$ By definition of $I_i$, we know that $\sum_{i} a_\ell s_\ell \in \sum_{i} I_i$. This shows that $1 \in \sum_{i} I_i$, which implies that $R=\sum_{i} I_i.$ 
    Conversely, suppose that $(2)$ holds. We will show that $\text{diam}(\Gamma(R,S)) \leq \text{diam}(\Gamma(R,U)) t.$ This will particularly show that $\Gamma(R,S)$ is connected. In fact, let $r \in R.$ Since $R=\sum_{s=1}^t I_{i_s}$, we can write  $r= \sum_{s=1}^t a_{i_s} x_{i_s}$, where $x_{i_s} \in K_{i_s}$ and $a_{i_s} \in R.$ Furthermore, since $\Gamma(R,U)$ is connected, each $a_{i_s}$ can be written as a sum of at most $\text{diam}(\Gamma(R,U))$ elements in $U$. Since each $K_{i_s}$ is stable under the action of $U$, this shows that $r$ can be written as a $\Z$-linear combination of at most $\text{diam}(\Gamma(R,U)) t$ elements in $S.$ By definition, $\text{diam}(\Gamma(R,S)) \leq \text{diam}(\Gamma(R,U)) t$.
\end{proof}

\begin{rem}
    Here, for simplicity, we assume that $\Gamma(R,U)$ is connected. This condition is not always satisfied. We refer the reader to \cite[Section 3]{unitary} and \cite[Theorem 3.6]{nguyen2024gcd} for an extensive analysis of \cref{prop:connectedness} in the \textit{maximal} case; namely $U=R^{\times}.$ 
    
    We also remark that the case $U=(R^{\times})^p$ where $p$ is a positive number that is invertible in $R$ is treated in \cite{podesta2021_finitefield}. In this case, the connectedness of $\Gamma(R, (R^{\times})^p)$ is directly related to the classical Waring problem. The more general case is treated in \cite{nguyen2024certain}. 
\end{rem}

\subsection{Primeness of $U$-unitary Cayley graphs}

A subset $X$ in a graph $G$ is called a homogeneous set if every vertex in $V(G) \setminus X$ is adjacent to either all or none of the vertices in $X$. Note that $X=V(G)$ as well as all vertex sets of size at most one are homogeneous sets. A homogeneous set $X$ with $2 \leq X < |V(G)|$ is called non-trivial. The graph $G$ is said to be prime if it does not contain any non-trivial homogeneous sets. We remark that the study of homogeneous sets is essentially equivalent  to the problem of decomposing a graph $G$ into smaller subgraphs (see \cite{barber2021recognizing,chudnovsky2024prime}). The later problem is a central problem in network theory, and it has found applications to dynamics of phase oscillators in networked systems (see \cite{jain2023composed, nguyen2022broadcasting}). 

In this section, we study the conditions for a $U$-unitary Cayley graph $\Gamma(R,S)$ be prime. Note that since a connected component of a graph is always a homogeneous set, we can safely assume that $\Gamma(R,S)$ is connected and anti-connected (meaning that its complement is connected). We have the following proposition, which is a direct generalization of \cite[Theorem 4.1]{chudnovsky2024prime}.
\begin{prop} \label{prop:homogeneous_ideal}
Suppose that $\Gamma(R,U)$ is connected. Suppose further that $\Gamma(R,S)$ is both connected and anti-connected. Then, the following conditions are equivalent. 
\begin{enumerate}
    \item $\Gamma(R,S)$ is not a prime graph. 
    \item There exists a proper ideal $I$; namely $I \neq 0$ and $I \neq R$, in $R$ such that $I$ is a homogeneous set in $\Gamma(R,S).$
\end{enumerate}
\end{prop}
\begin{proof}
    This statement can be proved using an identical argument as explained in \cite[Theorem 4.1]{chudnovsky2024prime}. For the sake of completeness, we briefly sketch it here. 

    By \cite[Theorem 3.4]{chudnovsky2024prime}, if $I$ is a maximal non-trivial homogeneous set of $\Gamma(R,S)$ containing $0$, then $I$ is a subgroup of $(R, +).$  We claim that $I$ is an ideal in $R$ as well. By \cref{prop:automorphism}, we know that  if $u \in U$, then the multiplication by $u$ is an automorphism of $\Gamma(R,S).$ Consequently, $uI$ is also a homogeneous set. Since $0 \in I \cap uI$, we conclude that $I \cup uI$ is also a homogeneous set. By the maximality of $I$, we must have $uI=I.$ We claim if $r \in R$, then $rI=I$. In fact, since $\Gamma(R,U)$ is connected, we can write
$r = \sum_{i=1}^d m_i u_i,$ where $m_i \in \Z$ and $u_i \in U.$ For each $h \in I$, we have  $rh = \sum_{i=1}^d m_i (u_i h).$
 Since $u_i h \in I$ and $I$ is a subgroup of $(R, +)$, we conclude that $rh \in I.$ This shows that $I$ is an ideal in $R$. 
\end{proof}

We discuss some conditions on the existence of an ideal which is simultaneously a homogeneous set in $\Gamma(R,S).$ First, we remark that an ideal $I$ in $R$ is necessarily stable under the action of $S.$ Therefore, $I$ is a union of various orbits $K_i.$ 

\begin{prop} \label{prop:homogeneous}
Suppose that $I$ is an ideal which is also a homogeneous set in $\Gamma(R,S).$ Let $K \in \{K_1, K_2, \ldots, K_m \}$ be an orbit. Then the following properties hold.

\begin{enumerate}
    \item If $K \subset S$ and $K \not \subset I$ then $K + I \subset S.$
    \item If $K \not \subset S$ and $K \not \subset I$ then $K+ I \subset R \setminus S.$
\end{enumerate}
Conversely, if $I$ is an ideal satisfying both of the above conditions then $I$ is a homogeneous set in $\Gamma(R,S).$
\end{prop}

\begin{proof}
    We will provide a proof for $(1)$ since $(2)$ follows from an identical argument (we can also use the fact that if $I$ is a homogeneous set in $\Gamma(R,S)$, then $I$ is also a homogeneous set in its complement, which is also a $U$-unitary Cayley graph by \cref{prop:complement_U}).

    Suppose that $K \subset S$ and $K \not \subset I$. Let $k \in K.$ Since $k$ is adjacent to $0$, it is adjacent to all vertices in $I.$ This shows that for each $m \in I$, $k$ is adjacent to $-m$ (since $-m \in I$). By definition, this shows that $k+m \in S.$ Because this is true for all $k \in K$ and $m \in I$, we conclude that $K+I \subset S.$

    The converse statement follows directly from the definition of a homogeneous set. 
\end{proof}

In particular, when $S=U$, since $U \cap I = \emptyset$, we have the following corollary, which is first proved in \cite{chudnovsky2024prime}.

\begin{cor} \label{cor:homogeneous_unit} (see \cite[Proposition 4.4, Proposition 4.7]{chudnovsky2024prime})
Suppose that $I$ is an ideal which is also a homogeneous set in $\Gamma(R,U).$ Then $I+U = U.$ Furthermore, $I$ is a nilpotent ideal in $R.$
\end{cor}

\begin{rem}
    The condition that $U+I =U$ is equivalent to a \textit{weaker} condition; namely $1+I \subset U.$ In fact, if $I+U =U$ then clearly $1+I \subset U$ since $1 \in U.$ Conversely, suppose that $1+I \subset U.$ Let $u \in U$ and $m \in I$, we claim that $u+m \in U.$ In fact, we can rewrite $u+m = u(1+u^{-1}m)$. Since $u^{-1}m \in I$, we know that $1+u^{-1}m \in U.$ Since $U$ is a group, $u(1+u^{-1}m)$ belongs to $U$ as well. 
\end{rem}

\section{Suppercharacter theory of a Frobenius ring and applications to spectra of $U$-Unitary Cayley graphs} \label{sec:frobenius}

\subsection{Supercharacter theories of a finite Frobenius ring}
In \cref{sec:U_unitary}, we show that for each choice of $U \subset R^{\times}$, there is a natural partition of $R$ into superclasses. Namely, these superclasses are precisely the orbits of $R$ under the natural action of $U.$ In order to apply the results for graph spectra in \cref{sec:general_theory}, we also need a compatible partition of the dual group of $(R,+).$ While it is unclear how to achieve this in full generality, there is a class of rings in which such a partition naturally exists; namely the class of finite commutative Frobenius rings. It is worth noting that many rings with arithmetic origins are Frobenius rings. Examples include finite fields, finite quotients of Dedekind domains, finite chain rings, and the group ring of an abelian group with coefficients in a finite field. In \cite{nguyen2024gcd}, we show that every finite commutative ring is a quotient of a Frobenius ring. We also provide various constructions of Frobenius ring extensions.

We now fix a Frobenius ring $R$. As explained in \cite{honold2001characterization, nguyen_pst}, being Frobenius means that $R$ is a $\Z/n$ algebra equipped with a non-degenerate $\Z/n$ linear functional $\psi \colon R \to \Z/n.$ Here, non-degenerate means that the kernel of $\psi$ does not contain any non-zero ideal in $R.$ Let $\zeta_n:= e^{\frac{2 \pi \bm{i}}{n}}$ be a fixed primitive $n$th-root of unity and $\chi\colon R \to \C^{\times}$ be the character defined by $\chi(a) =\zeta_n^{\psi(a)}.$ By \cite[Proposition 2.4]{nguyen2024integral}, the dual group $\Hom(R, \C^{\times})$ is a cyclic $R$-module generated by $\chi$; namely every character of $R$ is of the form $\chi_r$ where $\chi_r(a)=\chi(ra).$ By definition, the following identity holds for all $x, y \in R.$
\[ \chi_x(y) = \chi_y(x) = \chi_{xy}(1) =\chi(xy).\]

We now show that each subgroup of $R^{\times}$ gives rise to a supercharacter theory for $R.$

\begin{thm} \label{thm:induced_supercharacters}
    Let $U$ be a subgroup of $R^{\times}$ such that $-1 \in U.$ Let $\mathcal{K} = \{K_1, K_2, \ldots, K_m\}$ be the orbits of $R$ under the action of $U.$ Additionally, let $\mathcal{X} = \{X_1, X_2, \ldots, X_m\}$ be the partition of the character group of $R$ defined by 
    \[ X_i = \{\chi_{x} \mid x \in K_i \}.\]
Then the pair $(\mathcal{K}, \mathcal{X})$ is a symmetric supercharacter theory for $R.$ Furthermore, $(\mathcal{K}, \mathcal{X})$ satisfies Condition $4$ in \cref{def:supercharacter}.
\end{thm}

\begin{proof}
    The first two conditions in \cref{def:supercharacter} are automatic. For the third and fourth conditions, we adopt the following convention. 
Fix a set of representatives $\{r_1, r_2, \ldots, r_m \}$ for $\{K_1, K_2, \ldots, K_m\}$. 
Let $r_j'$ be an arbitrary element in $K_j.$ Then $r_j'=s r_j$ for some $s \in U.$ We have
\[
\sigma_i(r'_j)=\sum_{\psi\in X_i}\psi(r'_j)=\sum_{x\in K_i}\chi_x(sr_j)=\sum_{x\in K_i}\chi_{sx}(r_j)=\sum_{x\in K_i}\chi_x(r_j)=\sum_{\psi\in X_i}\psi(r_j)=\sigma_i(r_j).
\]
This shows that $\sigma_i$ is constant on $K_j$ for each $1 \leq j \leq m.$ In other words, the third condition in \cref{def:supercharacter} is satisfied. 

Now we check the fourth condition in \cref{def:supercharacter}. Let $\omega$ be an arbitrary element in $X_j$. Then $\omega=\chi_{x_j}$ for some $x_j\in K_j$. We can write $x_j= sr_j$ for some $s\in U$. Hence $\psi=\chi_{sr_j}$. We have
\[
\Omega_i(\omega)= \sum_{k\in K_i}\omega(k)=\sum_{k\in K_i}\chi_{sr_j}(k)=\sum_{k\in K_i}\chi_{r_j}(sk)=\sum_{k\in K_i}\chi_{r_j}(k)=\Omega(\chi_{r_j}).
\]
This shows that $\Omega_i(\omega)$ does not depend on the choice of $\omega$ in $X_j$, for each $1 \leq j \leq m.$ 

Finally, since $-1 \in U$, $K_i = -K_i$ for each $1 \leq i \leq m.$ By definition, $(\mathcal{K}, \mathcal{X})$ is symmetric. 
 \end{proof}

We now discuss several properties of this supercharacter theory which generalize various results in \cite{brumbaugh2014supercharacters}. 

\begin{prop} \label{prop:properties_characters}
    Let $U$ be a subgroup of $R^{\times}$ and $(\mathcal{K}, \mathcal{X})$ the supercharacter theory associated with $U$ as described in \cref{thm:induced_supercharacters}. Keeping the same notation as in the proof of \cref{thm:induced_supercharacters}. Then the following relations hold 
    \begin{enumerate}
        \item $\Omega_j(X_i)= \sigma_j(K_i). $
        \item We have 
        \[ \sigma_i(X_j)|K_j| = \sigma_j(X_i) |K_i|. \] 
        Note that this identity is a natural generalization of \cite[Lemma 1]{brumbaugh2014supercharacters}.
        \item Let $f\colon R \to \C$ be a superclass function. Let $\widehat{f}$ be the non-normalized super-Fourier transform of $f$ described in \cref{eq:fourier}. Then \cref{eq:explicit Fourier} can be rewritten as  
        \[ \widehat{f}(K_i) = \sum_{\ell=1}^m f(K_{\ell}) \sigma_{\ell}(K_i) = \sum_{\ell=1}^m f(K_{\ell}) \overline{\sigma_{\ell}(K_i)}. \]
        In other words 
        \[ \widehat{f} = \sum_{\ell =1}^m f(K_\ell) \sigma_\ell. \]
        We remark that this formula is a generalization of \cite[Equation 3.2]{brumbaugh2014supercharacters}.
        \item Let $\Sp$ be the space of all superclass functions on $R$ associated with $(\mathcal{K}, \mathcal{X}.)$ Let $\mathcal{F}\colon \Sp \to \Sp$ be the super-Fourier transform. Then $\mathcal{F}^2(f)=|R|f$ for all $f \in \mathcal{S}.$
    \end{enumerate}
\end{prop}

\begin{proof}
 Let us prove the first statement.   By definition, we have 
    \[ \Omega_j(X_i) = \sum_{x \in K_j} \chi_{r_i}(x) = \sum_{x \in K_j} \chi_x(r_i) = \sum_{\psi \in X_j} \psi(r_i) = \sigma_j(K_i). \]
    This proves the first statement. The second statement follows from the first part,  \cref{prop:relation}, and the fact that $|K_i|=|X_i|$ and $|K_j|=|X_j|$. The third statement is a consequence of the first statement and the formula for $\widehat{f}$ described in \cref{eq:explicit Fourier}. Finally, last statement follows from an identical to the one given in \cite{brumbaugh2014supercharacters}. 
\end{proof}

We now discuss a method to systematically calculate $\sigma_{\ell}(K_i).$ To do so, we introduce the following convention. Let $V_x$ (respectively $U_x$) be the kernel (respectively the image) of $U$ under the canonical map $R^{\times} \to (R/\Ann_{R}(x))^{\times}$. Here, for each subset $M \subset R$, $\Ann_R(M)$ is the annihilator ideal of $M$; namely 
\[ \Ann_{R}(M) = \{y \in R \mid yM =0 \}.\]

We remark that $R/\Ann_{R}(x)$ is a Frobenius ring and $\chi_x$ is one of its generating characters (see \cite[Lemma 4.3]{nguyen2024gcd}). For convenience,  we introduce the following notation. 
\begin{definition} \label{def:generalized_ramanujans}
Let $\varphi \in \widehat{R}= \Hom(R, \C^{\times})$ be an additive character of $R.$ The generalized Ramanujan sum $c(R,U, \varphi)$ is defined as follows 
    \[ c(R, U, \varphi) = \sum_{u \in U} \varphi(u).\]
\end{definition}
\begin{rem}
    When $U=R^{\times}$ and $\varphi = \chi_g$ then $c(R,U, \chi_g)$ is precisely the Ramanujan sum $c(g, R)$ introduced in \cite[Definition 4.8]{nguyen2024gcd}. 
\end{rem}

We have the following proposition, which shows that $\sigma_{\ell}(K_i)$ can be calculated via a Ramanujan sum of certain quotient ring of $R.$

\begin{prop} \label{prop:formula_sigma}
    \[ \sigma_{\ell}(K_i) = \frac{|U_{r_\ell}|}{|U_{r_\ell r_j}|} \sum_{u \in U_{r_{\ell} r_j}} \chi_{r_l r_i}(u) = \frac{|U_{r_\ell}|}{|U_{r_\ell r_j}|} c(R/\Ann_{r_{\ell} r_i}(R), U_{r_{\ell}r_i}, \chi_{r_{\ell} r_i}). \] 
\end{prop}

\begin{proof}

For each $x \in R$, let $\text{Stab}(x)$ be the stabilizer of $x$. We then have  
\[ \text{Stab}(x) = \{u \in U| ux =x \} = \{u \in U| (u-1) \in \Ann_{R}(x) \} = V_x.  \]
We have \[ |\text{Stab}(x)| = |V_x| = \frac{|U|}{|U_x|}. \]
For each $y \in R$, we denote by $K_y$  the class $K_i$ such that $K_i$ contains $y.$
As a result, for $x, y \in R$ we have 
\[ \sum_{\chi \in K_x} \chi_x(y)=\frac{1}{|\text{Stab}(x)|} \sum_{u \in U} \chi_{ux}(y) = \frac{|U_x|}{|U|} \sum_{u \in U} \chi_{ux}(y) = \frac{|U_x|}{|U|}  \sum_{u \in U} \chi_{xy}(u). \]
Consequently, we have  
\begin{align*}
    \sigma_\ell(K_i) &  = \frac{|U_{r_\ell|}}{|U|}  \sum_{u \in U} \chi_{r_\ell r_i}(u) = \frac{|U_{r_\ell}|}{|U|} \frac{|U|}{|U_{{r_{\ell} r_i}}|} \sum_{u \in U_{r_\ell r_i}} \chi_{r_\ell r_i}(u) = \frac{|U_{r_\ell}|}{|U_{r_\ell r_i}|} c(R/\Ann_{r_{\ell} r_i}(R), U_{r_{\ell}r_i}, \chi_{r_{\ell} r_i}).
\end{align*}
\end{proof}

\subsection{Applications to $U$-unitary Cayley graphs}
In this section, we use the results developed in the previous parts to study $U$-unitary Cayley graph. By \cref{thm:induced_supercharacters}, we know that $U$ gives rises to a supercharacter theory $(\mathcal{K}, \mathcal{X}).$ Furthermore, by definition, a $U$-unitary Cayley graph is precisely a super-graph with respect to this supercharacter theory. By \cref{prop:eigenvalues}, \cref{prop:relation}, and \cref{prop:formula_sigma} we have the following. 

\begin{thm} \label{thm:main_spectra}
    Let $\Gamma(R,S)$ be a $U$-unitary Cayley graph. Then, the spectrum of $\Gamma(R,S)$ is the multiset $\{[\lambda_{r_i}]_{|K_r|} \}_{i=1}^m.$ Here 
    \[ \lambda_{r_i}  = \widehat{1}_{S}(K_i) = \sum_{K_\ell \subset S} \Omega_\ell(X_i) = \sum_{K_\ell \subset S} \sigma_\ell(K_i) = \sum_{K_\ell \subset S}\frac{|U_{r_\ell}|}{|U_{r_\ell r_i}|} c(R/\Ann_{r_{\ell} r_i}(R), U_{r_{\ell}r_i}, \chi_{r_{\ell} r_i}).\]
    In particular, $\Gamma(R,S)$ has at most $m$ distinct eigenvalues. Furthermore, $S$ is determined by the $m$-dimensional vector $(\lambda_{r_i})_{i=1}^m.$
\end{thm}

\subsubsection{\textbf{Rationality of the spectra of $U$-unitary Cayley graphs.}}
A graph is called integral if all of its eigenvalues are integers. In \cite{nguyen2024integral}, we classify all integral Cayley graphs defined over a finite Frobenius ring. We provide here a slight generalization. 
\begin{prop} \label{prop:rationality}
Let $\Gamma(R,S)$ be an $U$-unitary Cayley graph.    Let $H$ be the preimage of $U$ under the canonical map $(\Z/n)^{\times} \to R^{\times}.$ Then the eigenvalues of $\Gamma(R,S)$ belong to the fixed  field $\Q(\zeta_n)^{H}$. In particular, if $H = (\Z/n)^{\times}$ then all eigenvalues of $\Gamma(R,S)$ are integers, i.e., $\Gamma(R,S)$ is an integral graph. 
\end{prop}
\begin{proof}
    By \cref{prop:eigenvalues}, we know that the spectrum of $\Gamma(R,S)$ is the multiset $\{\lambda_r\}_{r \in R}$ where 
$\lambda_r = \sum_{s \in S} \chi_r(s).$ We remark also that $\lambda_r \in \Q(\zeta_n).$ Furthermore, the Galois group $\Gal(\Q(\zeta_n)/\Q)$ is precisely $(\Z/n)^{\times}$. More precisely, there exists an isomorphism $\Phi\colon (\Z/n)^{\times} \to \Gal(\Q(\zeta_n)/\Q)$ defined by $a \mapsto \sigma_a$ where $\sigma_a$ is defined by the rule $\sigma_a(\zeta_n) = \zeta_n^a.$ 
We can see that for each character $\chi_r$ of $R$, we have $\sigma_a(\chi_r(s)) = \chi_{ar}(s)=\chi_r(as).$ Let $h \in H$, then we have 
\[ 
\sigma_h(\lambda_r) = \sum_{s \in S} \sigma_h(\chi_r(s)) = \sum_{s \in S} \chi_r(hs) = \sum_{s \in S} \chi_r(s). \]
The last equality follows from the fact that $S$ is stable under the action of $U$ and hence is also stable under the action of $H.$ This shows that $\lambda_r$ belongs to the fixed field $\Q(\zeta_n)^{H}$ for each $r \in R.$  
\end{proof}

We study the converse of \cref{prop:rationality}. The naive guess does not quite work as a graph could be both $U_1$-unitary and $U_2$-unitary for two different $U_1, U_2$. For example, if $\Gamma(G,S)$ is $U$-unitary, then it is also $U_1$-unitary for any subgroup $U_1$ of $U.$  To avoid this tautological issue, we introduce the following definition. 

\begin{definition}
    A Cayley graph $\Gamma(R,S)$ is called purely $U$-unitary if $\Gamma(R,S)$ is not a $U'$-unitary Cayley graph for some $U' \subset R^{\times}$ such that $U \subsetneq U'.$
\end{definition}

The following lemma says that the induced action described above is the only way for a purely $U$-unitary Cayley graph to have a different unitary structure. 
\begin{lem} \label{lem:purely_U}
    Let $\Gamma(R,S)$ be a purely $U$-unitary Cayley graph. Suppose that $\Gamma(R,S)$ is also a $U_1$-unitary Cayley graph. Then $U_1 \subset U.$
\end{lem}
\begin{proof}
    By definition, $S$ is stable under the action of $U_1$ and $U.$ Consequently, it is stable under the action of the product $U_1 U.$ Since $\Gamma(R,S)$ is purely $U$-unitary, we must have $U_1U \subset U.$ This implies that $U_1 \subset U_2.$
\end{proof}

\begin{cor}
Let $\Gamma(R,S)$ be a Cayley graph. Then, there exists a unique subgroup $U \subset R^{\times}$ such that $\Gamma(R,S)$ is purely $U$-unitary. 
\end{cor}

\begin{proof}
    Let us define 
    \[ U = \{ u \in R^{\times} \mid uS = S .\}. \]
    We can check that $U$ is a subgroup of $R^{\times}.$ Furthermore, by definition of $U$ we know that $US=S$. Therefore, we conclude that $\Gamma(R,S)$ is $U$-unitary. Let $U_1$ be a subgroup of $R^{\times}$ such that $S$ is $U_1$-unitary. Then, for each $u_1 \in U_1$, we must have $u_1 S= S.$ By definition of $U$, we know that $u_1 \in U.$ Therefore, $U_1 \subset U.$ We conclude that $\Gamma(R,S)$ is purely $U$-unitary.   
\end{proof}

With these setups, we now generalize \cref{cor:integral_graph} to a broader statement concerning the rationality of the spectra of $\Gamma(R,S).$ We formally introduce the following definition. 

\begin{definition}
    Let $K$ be a subfield of $\C.$ A graph $G$ is called $K$-rational if all eigenvalues of $G$ belong to $K.$ Note that a graph is integral if and only if  it is $\Q$-rational. 
\end{definition}

We begin with the following observation, which generalizes \cite[Theorem 2.7]{nguyen2024integral}. Although the argument we present is similar to those in \cite{nguyen2024integral, so2006integral}, we will explain it within the context of supercharacter theory on $R.$

\begin{thm} \label{thm:general_rationality}
Let $T$ be a symmetric subset of $R$ and $\Gamma(R,T)$ the Cayley graph associated with $T$. Let $K$ be a subfield of $\Q(\zeta_n)$ and $H_1$ the group associated with $K$ under the Galois correspondence; namely $K = \Q(\zeta_n)^{H_1}.$  Then, the following conditions are equivalent.
\begin{enumerate}
    \item  $\Gamma(R,T)$ is $K$-rational. 
    \item $\Gamma(R,T)$ is an $U_1$-unitary Cayley graph where $U_1$ is the image of $H_1$ in $R^{\times}$ under the canonical map $(\Z/n)^{\times} \to R^{\times}.$

\end{enumerate}

\end{thm}

\begin{proof}
By \cref{prop:rationality} we know that $(2) \implies (1).$ Let us show that $1) \implies (2).$

Let $(\mathcal{K}_1, \mathcal{X}_1)$ be the supercharacter theory associated with $U_1.$ In particular, let $\mathcal{K}_1 =\{O_1, O_2, \ldots, O_s\}$ be the orbits of $R$ under the action of $U_1.$  Let ${v}_i \in K^{|R|}$ be the characteristic vector of $O_i$; namely 
\[ {v}_i[r] = 
\begin{cases}
    1, & \text{if } r \in O_i \\
    0, & \text{if } r \not \in O_i.
\end{cases}
\]
By definition, ${v}_1, {v}_2, \ldots, {v}_d$ are linearly independent over $K$. Let $\mathcal{V}$ be the $K$-vector space generated by the ${v}_i$'s.

Let $A_{R}=(\chi_r(t))_{r,t}$ be the DFT matrix associated with $R$ (see \cite{kanemitsu2013matrices}). 
Let us also define 
\[ \mathcal{A} = \{{v} \in K^{|R|} \mid A_{R} {v} \in K^{|R|}\} .\]
We first claim that $\mathcal{V} \subset \mathcal{A}.$ In fact, let ${v} \in \mathcal{V}$. Then $v$ can be written as a $K$-linear combination of $v_i$; namely
$v = \sum_{i=1}^d a_i v_i$. We then have $A_{R}v = \sum_{i=1}^d a_i A_{R}v_i.$ Let $\sigma_a \in \Gal(\Q(\zeta_n)/\Q)$ with $a \in H_1$. Since $\sigma(a_i)=a_i$, we have 
\[ \sigma_a(A_R(v)) = \sum_{i=1}^d a_i \sigma_a(A_R v_i). \]
By definition 
\[ \sigma_a(A_Rv_i) = \left(\sigma_a (\sum_{s \in O_i} \chi_r(s)) \right)_{r \in R} = \left(\sum_{s \in O_i} \chi_r(as) \right)_{r \in R} = \left(\sum_{s \in O_i} \chi_r(s) \right)_{r \in R} = A_Rv_i. \]
The second to the last equality follows from the fact that each $O_i$ is stable under the multiplication by $a \in H_1.$ 

Next, we claim that if $w \in \mathcal{A}$ then $A_{R}w \in \mathcal{V}.$ Let $w=(w_t)_{t \in R} \in K^{|R|}$ and $A_R w = (u_r)_{r \in R}.$ Because $w \in \mathcal{A}$, $u_r \in K$ for all $r$. By definition 
\[ u_r = \sum_{t \in R} \chi_r(t) w_t.\]
We claim that if $r_1,r_2$ belong to the same equivalence class then $u_{r_1} = u_{r_2}.$ In fact, since $r_1, r_2$ belong to the same equivalence class, we can find $a \in H_1$ such that $ar_1 = r_2.$ We then have 
\[ u_{r_1} = \sigma_a(u_{r_1}) =  \sum_{t \in R} \chi_{ar_1}(t) \sigma_a (w_t) =  \sum_{t \in R} \chi_{r_2}(t) w_t = u_{r_2}.\]
In summary, we have $\mathcal{V} \subset \mathcal{A}$, $A_{R} \mathcal{A} \subset \mathcal{V}$. Furthermore, since $A_R$ is invertible, we must have $\mathcal{A} = \mathcal{V}.$

We recall that the eigenvalues of $\Gamma(R,T)$ are precisely $A_R 1_{T}$ where $1_T$ is the characteristic vector of $T.$ By our assumptions, there eigenvalues belong to $K$, we conclude that $T \in \mathcal{V}.$ This shows that $T$ is a union of some classes in $\mathcal{K}_1.$ In other words, $\Gamma(R,S)$ is $U_1$-unitary.     
\end{proof}

\begin{thm} \label{thm:second_rationality}
    Let $\Gamma(R,S)$ be a purely $U$-unitary Cayley graphs. Let $K, H_1, U_1$ be as in \cref{thm:general_rationality}. Suppose further that $K$ is the smallest subfield of $\Q(\zeta_n)$ such that $\Gamma(R,S)$ is $K$-rational. Then, $U_1 = U.$
\end{thm}

\begin{proof}
 By \cref{thm:general_rationality}, since $\Gamma(R,S)$ is $K$-rational, it is $U_1$-unitary. Because $\Gamma(R,S)$ is purely $U$-unitary, \cref{lem:purely_U} implies that $U_1 \subset U$.

Conversely, let $H$ be the preimage of $U$ in $\Z/n^{\times}.$ \cref{prop:rationality} implies that $\Gamma(R,S)$ is $\Q(\zeta_n)^H$-rational. Because $K$ is the smallest subfield where $\Gamma(R,S)$ is $K$-rational, we must have $\Q(\zeta_n)^{H_1} \subset K := \Q(\zeta_n)^{H}.$ By the Galois correspondence, we have $H \subset H_1.$  We conclude that $U \subset U_1$ and hence $U_1 =U.$
\end{proof}
When $K=\Q$, we have the following simple corollary. 

\begin{cor} \label{cor:integral_graph}
Suppose that $\Gamma(R,S)$ is an integral graph. Suppose further that $\Gamma(R,S)$ is purely $U$-unitary. Let $H$ be as in \cref{prop:rationality}. Then $H = (\Z/n)^{\times}.$
\end{cor}

\begin{rem} \label{rem:group_ring}
    While we state \cref{thm:general_rationality} and \cref{thm:second_rationality} for finite Frobenius rings, they hold for a general abelian group as well. More precisely, suppose that $G$ is an abelian group of order $n.$ Then $G$ is a $\Z/n$-module and the group $(\Z/n)^{\times}$ acts naturally on $G$ by multiplication. Let $H_1, K$ be as in \cref{thm:general_rationality}. Then, we claim that $\Gamma(G,S)$ is $K$-rational if and only $S$ is stable under the action of $H_1.$

    In fact, there exists an isomorphism of abelian groups $\Phi: G \to \prod_{i} \Z/p_{i}^{\alpha_i}:=R$ where $p_i$ are prime numbers (not necessarily distinct). This isomorphism is compatible with the action of $(\Z/n)^{\times}$ on both sides. Furthermore, there exists an induced isomorphism of graphs $\Phi: \Gamma(G, S) \cong \Gamma(R, \Phi(S)).$ As a result, $\Gamma(R, \Phi(S))$ is $K$-rational if and only if $\Gamma(R,S)$ is. By \cref{thm:general_rationality} applied to $R$, which is a Frobenius ring, we know that this is equivalent to the condition that $\Phi(S)$ is stable under the action of $H_1.$ This in turn, is equivalent to the condition that $S$ is stable under the action of $H_1.$
\end{rem}

\subsection{\textbf{Primeness and eigenvalue $0$}}
In \cite[Corollary 4.3]{chudnovsky2024prime}, we show that if the Cayley graph $\Gamma(R,U)$ is not prime, then $0$ is an eigenvalue with multiplicity at least $\frac{|R|}{2}.$ The argument given there is graph theoretic in nature. Our reformulation of the eigenvalues allows us to provide a more concrete statement. 
\begin{prop} \label{prop:0_multiplicity}
Suppose that $I$ is a non-zero ideal and a homogeneous set in $\Gamma(R,U)$. Let $r \in R$ such that $r \notin \Ann_{R}(I)$. Then $\lambda_r = 0$. Consequently, $0$ is an eigenvalue of $\Gamma(R,S)$ with multiplicity at least $|R|\left(1-\frac{1}{|I|}\right)$, which is at least $\frac{|R|}{2}$.
\end{prop}
\begin{proof}
    By \cref{prop:homogeneous} we know that $U+I \subset U$. Consequently, for cardinality reasons, $m+U = U$ for each $m \in I.$ Let $r \in R$ such that $r \not \in \Ann_{R}(I).$ This shows that $rI$ is a non-zero ideal of $R$. Because $\chi$ is non-degenerate, there exists $m \in I$ such that $\chi(rm)=\chi_r(m) \neq 1.$ We then have 
    \[ \lambda_r = \sum_{u \in U} \chi_r(u) = \sum_{u \in U} \chi_r(m+u) = \chi_r(m) \sum_{u \in U} \chi_r(u) =\chi_r(m) \lambda_r.\]
    Because $\chi_r(m) \neq 1$, we conclude that $\lambda_r=0.$

    Since $I$ is a non-zero ideal in $R$, $\Ann_{R}(I)$ is a proper ideal of $R$ as well. In fact, by \cite[Theorem 1]{honold2001characterization}, we have the following identity  $|I| |\Ann_{R}(I)|=|R|.$ This shows that the set $R \setminus \Ann_{R}(I)$ has exactly $|R|(1-\frac{1}{|I|})$ elements. In particular, this show that $\lambda_r=0$ for at least $\frac{|R|}{2}$ values of $r.$
\end{proof}
In many examples that we study, the converse of \cref{prop:0_multiplicity} is also true. For this reason,  we propose it as an open question for further investigation.
\begin{question} \label{question:eigen_0}
Suppose that $\Gamma(R,U)$ is both connected and anti-connected. Is it true that $\Gamma(R,U)$ is prime if and only if $0$ is not an eigenvalue? 
\end{question}

We provide some partial answers for \cref{question:eigen_0}.

\begin{prop}
    Suppose that $U=R^{\times}.$ Suppose that $\Gamma(R,R^{\times})$ is connected and anti-connected. Then, $\Gamma(R,R^{\times})$ is prime if and only if $0$ is not an eigenvalue of $\Gamma(R,R^{\times}).$
\end{prop}

\begin{proof}
One direction follows from \cref{prop:homogeneous} and \cref{prop:0_multiplicity}. Let us now prove the converse; namely if $\Gamma(R,R^{\times})$ is prime then $0$ is not an eigenvalue of $\Gamma(R,R^{\times}).$ In fact, by \cite[Theorem 4.34]{chudnovsky2024prime}, $R$ must be a product of fields. Consequently, $\Gamma(R,R^{\times})$ is a tensor product of complete graphs. As a result, their eigenvalues are non-zero (see also \cite[Proposition 10.2]{unitary}). 
\end{proof}

\begin{rem}
    As noted in \cite{chudnovsky2024prime}, the condition at that $\Gamma(R,R^{\times})$ is both connected and anti-connected is important. In fact, if we take $R$ to be a field, then $\Gamma(R, R^{\times})$ is a complete graph which is not prime unless $|R|=2.$ However, in this case, its eigenvalues are $|R|-1$ and $-1$ which are not zero. 
\end{rem}

We discuss another instant where we can prove \cref{question:eigen_0}. 
\begin{prop} \label{prop:prime_to_p}
    Let $R$ be a $\F_p$-algebra. Let $U \subset R^{\times}$ such that $p \nmid |U|.$ Then, the following statements hold.
    \begin{enumerate}
        \item $0$ is not an eigenvalue of $\Gamma(R,U).$
        \item There is no non-zero ideal $I$ such that $I$ is a homogeneous set in $\Gamma(R,S).$ In particular, if $\Gamma(R,U)$ is connected and anti-connected then $\Gamma(R,U)$ is prime. 
    \end{enumerate}
In particular, if $R$ is a field then these statements hold. 
\end{prop}

\begin{proof}
  For the first statement, we need to show that if $r \in R$ then $\lambda_r = \sum_{u \in U} \chi_r(u) \neq 0.$ By definition 
  \[ \sum_{u \in U} \chi_r(u) = \sum_{u \in U} \zeta_p^{\psi_r(u)} = \sum_{\ell=0}^{p-1} \left[\sum_{\psi_r(u)=\ell} \zeta_p^{\ell} \right] = \sum_{\ell=0}^{p-1} a_\ell \zeta_p^{\ell}. \]
  Here $\psi \colon R \to \F_p$ is the non-degenerate $\F_p$-linear function on $R$ is described at the beginning of this section and $a_{\ell}$ is the number of $u \in U$ such that $\psi_r(u)=\ell$. Assuming that $\lambda_r=0.$ Then we must have $a_0 = a_1 = \cdots = a_{p-1}.$ This would imply that $|U| = \sum_{\ell=0}^{p-1} a_\ell = pa_0$ which is a contradiction.

  Let us prove the second statement. Suppose $I$ is a non-zero ideal which is a homogeneous set in $\Gamma(R,S)$. By \cref{prop:homogeneous}, we know that $I + U \subset U.$ In particular, we have $1+I \subset U.$ By \cref{cor:homogeneous_unit}, $I$ is a nilpotent ideal. Since $R$ has characteristics $p$, for each $m \in I $ we have $(1+m)^p = 1+m^p$. This shows that $1+I$ is a nontrivial $p$-group. Since $p \nmid |U|$, this is impossible. 
\end{proof}

Here is another example where can verify \cref{question:eigen_0}. This is an example of a generalized Paley graph which we will discuss in more details in \cref{subsec:paley}. Here, we focus on a particular case that is discussed in \cite{minac2023paley}.

\begin{prop}
    Let $m$ be a positive integer such that all prime divisors of $m$ are of the form $4k+1.$ Let $\rho\colon R:=(\Z/m) \to \C$ be the multiplicative character defined by Jacobi symbol $\rho(a) = \left(\frac{a}{m} \right)$ (see \cite{minac2023paley}). Let $U = \ker(\rho) = \{u \in (\Z/m)^{\times} \mid \rho(u)=1 \}$. 
    The following conditions are equivalent. 
    \begin{enumerate}
        \item $m=m_0$ where $m_0 = \rad(m)$ where $\rad(m)$ is the product of all distinct prime divisors of $m.$
        \item $0$ is not an eigenvalue of $\Gamma(\Z/m,U).$
        \item $\Gamma(\Z/m, U)$ is  prime. 

    \end{enumerate}
\end{prop}

\begin{proof}
    By the definition of $\rho$, $\rho$ is periodic modulo $m_0$; namely $\rho(a)= \rho(a+m_0).$ As a result, $m_0R$ is a homogeneous set in $\Gamma(R, U)$ by \cref{prop:homogeneous}. This shows that $(3) \implies (1).$ Furthermore, by \cref{prop:0_multiplicity}, we know that $(2) \implies (1).$ Let us show that $(1) \implies (2).$ In fact, if $m=m_0$, then the eigenvalues of $\Gamma(R, U)$ are calculated in \cite[Theorem 3.5]{minac2023paley}, which are all non-zero. 
    Let us now show that $(1) \implies (3).$ First, since $1 \in U$, we know that $\Gamma(R,U)$ is connected. Furthermore, by the Chinese remainder therem, there exists $a \in R^{\times}$ such that $\rho(a) = -1$ and hence $a \not \in U.$ This shows that the complement of $\Gamma(R,U)$ is also connected. Suppose to the contrary that  $\Gamma(R,U)$ is not prime. Then by \cref{prop:homogeneous_ideal} and \cref{cor:homogeneous_unit}, there exists a non-zero nilpotent ideal $I$ such that $1+I \subset U.$ However, since $m_0 = m$, there are no non-zero nilpotent ideals in $R$. This leads to a contradiction. We conclude that $\Gamma(R,U)$ must be prime. 
\end{proof}

Finally, we discuss the answer for \cref{question:eigen_0} for $p$-unitary Cayley graphs. 
\begin{prop}
    Let $p$ be a prime number and $R$ a finite Frobenius ring such that $p$ is invertible in $R.$ Let $U=(R^{\times})^p$.  Let $G_R(p)$ be the $p$-unitary Cayley graph $\Gamma(R,U)$. Suppose that $G_R(p)$ is connected and anticonnected. Then, the following statements are equivalent 
    \begin{enumerate}
        \item $\Gamma(R,U)$ is prime. 
        \item $R$ is a product of fields. 
        \item $0$ is not an eigenvalue of $\Gamma(R,U).$
    \end{enumerate}
\end{prop}

\begin{proof} We will show that $(1) \implies (2) \implies (3) \implies (1).$

Let us first show that $(1) \implies (2)$. Let $R=R_1 \times R_2 \times \ldots \times R_d$ be the factorization of $R$ into a product of local rings. By \cite[Proposition 3.1]{nguyen2024certain}, we know that the Jacobson radical of $R$ is a homogeneous set in $\Gamma(R,U).$ Consequently, if $\Gamma(R,U)$ is prime, then $\Rad(R)=0.$ In other words, $R$ must be a product of fields. 

Let us now show that $(2) \implies (3).$ In fact, in this case, $G_R(p)$ is isomorphic to the tensor products of $G_{R_i}(p).$ Since $R_i$ is a field, $U_i = (R_i^{\times})^p$ satisfies the conditions of \cref{prop:prime_to_p}. As a result, all eigenvalues of $G_{R_i}(p)$ are non-zero. This shows that all eigenvalues of $R_i$ are non-zero as well. 

Finally, \cref{prop:homogeneous_ideal} and \cref{prop:0_multiplicity} show that $(3) \implies (1).$ 
\end{proof}

\subsection{Some examples.}
We discuss various examples of well-known $U$-unitary Cayley graphs over a finite Frobenius ring with a fixed generating character $\chi$. In particular, we explain the relationship between their spectra and various arithmetical sums.  

\subsubsection{\textbf{Gcd-graphs}}
The case $U=R^{\times}$ is considered in various works such as \cite{unitary, klotz2007some, nguyen2024gcd}. In this case, the generalized Ramanujan sum defined in \cref{def:generalized_ramanujans} can be explicitly calculated. More precisely, in \cite[Theorem 4.14]{nguyen2024gcd}, we show that 
\[ c(R, R^{\times}, \chi) = c(g,R) = \frac{\varphi(R)}{\varphi(R/\Ann_{R}(g))} \mu(R/\Ann_R(g))\]
Here $\varphi$ is the generalized Euler function; $\varphi(R)=|R^{\times}|$ and $\mu$ is the generalized M\"obius function. We note that this sum is independent of the choice of the generating character $\chi$ (since two generating characters differ by a scaling by a unit in $R$.) 
\subsubsection{\textbf{$p$-unitary Cayley graphs}}
Let $U=(R^{\times})^p$ where $p$ is a positive integer. In this case, the graph $\Gamma(R, U)$ is called a $p$-unitary Cayley graph (see \cite{nguyen2024certain,podesta2019spectral, podesta2023waring}.) In this case, the generalized Ramanujan sum $c(R, U, a)$ is related to the Heilbronn sum, which we now recall. 

\begin{definition}
Let $p$ be a prime number. For each $a \in \Z/p^2$ The Heilbronn sum $H_a(p)$ is defined as follows 
\[ H_p(a) = \sum_{\ell =1}^{p-1} e^{\frac{2 \pi i a \ell^p}{p^2}}.\]
\end{definition}
By \cite[Lemma 3.1]{garcia2018supercharacter}, we know that for $R=\Z/p^2$, $\{\ell^p\}_{\ell =1}^{p-1}$ is a set of representative for $U = (R^{\times})^p$. Furthermore, if we let $\chi$ be the generating character of $\Z/p^2$ defined by $\chi(a) = e^{\frac{2 \pi i a}{p^2}}$, then we can rewrite
\[ H_p(a) = \sum_{u \in U} \chi_a(u) = c(\Z/p^2, U, \chi_a).\]
In summary, we have the following.
\begin{prop}
    Let $R=\Z/p^2$. Then the spectrum of the $U$-unitary Cayley graph $\Gamma(R, U)$ is precisely the multiset $\{H_a(p)\}_{a \in R}$ of Heilbronn sums. 
\end{prop}

\subsubsection{\textbf{Paley graphs}} \label{subsec:paley}

We recall that the Paley graph $P_{\rho}$ is $U$-unitary Cayley graph $\Gamma(R,U)$ with 
\[ U = \ker(\rho) =\{u \in R^{\times} \mid \rho(u)=1\}. \]
Here $\rho\colon R^{\times} \to \C^{\times}$ is a multiplicative character.

\begin{definition}(See \cite[Definition 1]{lamprecht1953allgemeine})
    Let $\rho\colon R^{\times} \to \C^{\times}$ is a multiplicative character and $\varphi\colon R \to \C^{\times}$ an additive character of $R.$  Following a standard convention, we extend $\rho$ to be a function $\rho\colon  R \to \C$ by the rule that $\rho(a)=0$ if $a \not \in R^{\times}.$ The Gauss sum  $\tau(\rho, \varphi)$ is defined to be 
    \[ \tau(\rho, \varphi) = \sum_{r \in R^{\times}} \rho(r) \varphi(r) = \sum_{r \in R} \rho(r) \varphi(r).\]
\end{definition}

\begin{definition}
    We say that $\rho$ has order $d$ if $\rho^d(u)=1$ for all $u \in R^{\times}.$
\end{definition}

\begin{prop}
Suppose that $\rho$ has order $d.$ Then,    $P_{\rho}$ is $K$-rational where $K$ is a number field of degree at most $d.$
\end{prop}

\begin{proof}
    Let $H$ be as in \cref{prop:rationality}. Then $H$ is the kernel of the induced Dirichlet character $\rho\colon (\Z/n)^{\times} \to \mu_d \subset  \C^{\times}$. Consequently, the fixed field $K=\Q(\zeta_n)^H$ has degree at most $d.$ By \cref{prop:rationality}, we know that $P_{\rho}$ is $K$-rational. 
\end{proof}
\begin{prop} \label{prop:characteristic_function}
Suppose that $\rho\colon R^{\times} \to \C^{\times}$ is a multiplicative character of order $d.$ Let $U = \ker(\rho).$ Then, the characteristic function of $U$ is given by 
\[ 1_{U}(a) = \frac{1}{d} \cdot \frac{1-\rho(a)^d}{1-\rho(a)} \rho(a) = \frac{1}{d} \sum_{i=1}^d \rho^i(a).\]
\end{prop}

\begin{proof}
    By definition, $\rho(a)=0$ or $\rho(a)^d=1$. Consequently, we have 
\[
        \frac{1}{d}\frac{1-\rho(a)^d}{1-\rho(a)} \rho(a) = 
        \begin{cases}
        1 \text{ if } \rho(a)=1 \\ 
        0 \text{ else.}      
        \end{cases}
\] 
We conclude that $1_{U}(a) = \dfrac{1}{d}\dfrac{1-\rho(a)^d}{1-\rho(a)} \rho(a) $ for all $a \in R.$
\end{proof}
We can describe the Ramanujan sums in \cref{def:generalized_ramanujans} via Gauss sums (see \cite[Theorem 3.5]{minac2023paley} for an explicit calculation in the case $\rho$ is a quadratic character modulo $n$). 

\begin{prop}
Let $\varphi\colon R \to \C$ be an additive character of $R$ and $\rho\colon R^{\times} \to \C$ a multiplicative character of order $d.$ Let $c(R, \ker(\rho), \varphi)$ be the generalized Ramanujan sum described in \cref{def:generalized_ramanujans}. Then 
\[ c(R, \ker(\rho), \varphi) = \frac{1}{d} \sum_{i=1}^d \tau(\rho^i, \varphi).\]
\end{prop}

\begin{proof}
    By definition, we have 
    \begin{align*}
    c(R, \ker(\rho), \varphi) &= \sum_{u \in \ker(\rho} \varphi(u)  = \sum_{u \in R} 1_U(a) \varphi(u)= \sum_{u \in R} \left[\sum_{i=1}^d \frac{1}{d} \rho^i(u) \right] \varphi(a) \\ 
    &= \frac{1}{d}\sum_{i=1}^d  \sum_{u \in R} \rho^i(u) \varphi(u) =  \frac{1}{d} \sum_{i=1}^d \tau(\rho^i, \varphi).
    \end{align*}
    \qedhere
\end{proof}

\section*{Acknowledgements}
The first-named author would like to thank Professor Stephan R. Garcia for his helpful correspondence. He is also grateful to Professor Torsten Sander for sharing his insights and providing constant encouragement, which has been instrumental in advancing this line of research to its current stage.
  The second-named author gratefully acknowledges the Vietnam Institute for Advanced Study in Mathematics (VIASM) for their hospitality and support during a visit in 2025. Additionally, the papers \cite{brumbaugh2014supercharacters,fowler2014ramanujan} have greatly influenced our work. We thank these authors for their invaluable ideas and contributions.


\end{document}